\documentclass[12pt]{article}
\usepackage{geometry}                
\usepackage{graphicx,color,mathtools}
\usepackage{amssymb,amsmath,amsthm,mathrsfs}
\usepackage[all,cmtip]{xy}
\usepackage{epstopdf, comment, url}
\usepackage{bm} 
\usepackage{enumitem}

\usepackage[pdftex,bookmarks,pdfnewwindow,plainpages=false,unicode,pdfencoding=auto]{hyperref}

 \hypersetup{
pdfauthor={Oleg Ivrii},
pdftitle={Non-tangential ranges of holomorphic functions at Plessner points},
pdfsubject={Potential theory},
bookmarksdepth={4}
}

\numberwithin{equation}{section}

\newtheorem{theorem}{Theorem}[section]

\newtheorem{lemma}[theorem]{Lemma}
\newtheorem{corollary}[theorem]{Corollary}

\theoremstyle{remark}
\newtheorem*{remark}{Remark}

\theoremstyle{definition}

\DeclareMathOperator{\Mdim}{M.dim }

\DeclareMathOperator{\Hdim}{H.dim }

\DeclareMathOperator{\osc}{osc}
\DeclareMathOperator{\capacity}{cap}

\title{Non-tangential ranges of holomorphic functions \\ at Plessner points}
\author{Oleg Ivrii}
\date{September 16, 2026}

\begin{document}

\maketitle

\begin{abstract}
Consider the random lacunary series $f(z) = \sum_{k=1}^\infty \frac{\xi_k}{\sqrt{k}} \, z^{2^k}$ on the unit disk, where $\{ \xi_k \}$ are independent standard complex Gaussian random variables. We show that almost surely, a.e.~$\zeta \in \partial \mathbb{D}$ is a Plessner point of $f$, yet the image of every Stolz angle with vertex at $\zeta$ has asymptotic density zero. This gives a negative answer to questions of Collingwood and Baernstein concerning possible strengthenings of Plessner's theorem.

In this example, for a.e.~$\zeta \in \partial \mathbb{D}$, the non-tangential range of $f$ at $\zeta$ has zero area.
The non-tangential range cannot be much smaller: we show that for an arbitrary holomorphic function on the unit disk, the non-tangential range has Hausdorff dimension 2 at almost every Plessner point.
\end{abstract}

\section{Introduction}

Let $f$ be a holomorphic function on the unit disk $\mathbb{D} = \{ z \in \mathbb{C} : |z| < 1 \}$.
A point $\zeta \in \partial \mathbb{D}$ is called a {\em Fatou point} of $f$ if $f$ has a finite
non-tangential limit at $\zeta$, and a {\em Plessner point} if the image of every truncated
Stolz angle with vertex at $\zeta$ is dense in $\mathbb{C}$. Plessner's
theorem asserts that a.e.~point on the unit circle is either a Fatou point or a Plessner point. See the original paper  \cite{plessner} or
any of the books \cite{collingwood-lohwater, garnett-marshall, pommerenke}.

Collingwood asked whether Plessner's theorem remains true if the condition that ``the image
of every truncated Stolz angle is dense in $\mathbb{C}$'' is replaced by the stronger condition ``the image of every truncated
Stolz angle omits only a set of area zero.'' Later, Baernstein asked the same question with area replaced by logarithmic capacity.
His question is the stronger of the two, since a set of zero capacity has zero
area. These questions are recorded as Problems 5.20 and 5.57 in the fiftieth anniversary edition of Hayman's list
\cite{hayman}, while Collingwood's problem appeared among the 141 problems in the original 1967 list \cite{hayman-original}.

In this paper, we show that the answers to both questions are negative. Our counterexample is a random lacunary series
\begin{equation}
\label{eq:f-def}
f(z) = \sum_{k=1}^\infty \frac{\xi_k}{\sqrt{k}} \, z^{2^k},
\end{equation}
where $\xi_k \sim \mathcal N_{\mathbb{C}}(0,1)$ are independent standard complex Gaussian random variables. 

For $0<\alpha<\pi$, let $R_\alpha(f,\zeta)$ denote the set of values $a\in\mathbb C$ that $f$ attains along a sequence of points tending to $\zeta$ inside the Stolz angle with vertex at $\zeta$ and opening $\alpha$.
We define the {\em non-tangential range} of $f$ at $\zeta$ as $$R_{\Delta}(f, \zeta) = \bigcup_{0 < \alpha < \pi} R_\alpha(f, \zeta).$$
From the definitions, it follows that $R_\alpha(f,\zeta)$, $0 < \alpha < \pi$, are $G_\delta$ sets and $R_{\Delta}(f, \zeta)$ is a $G_{\delta \sigma}$ set.

\begin{theorem}
\label{main-thm}
Almost surely, the random lacunary series (\ref{eq:f-def}) converges absolutely and uniformly on compact subsets of the unit disk and defines a holomorphic function with the following properties:

{\em (i)} For a.e.~$\zeta\in\partial\mathbb D$, we have
$\liminf_{r \to 1} |f(r\zeta)| = 0$ and $\limsup_{r \to 1} |f(r\zeta)| = \infty$.
Consequently, by Plessner's theorem, a.e.~point on the unit circle is a Plessner point.

{\em (ii)} For a.e.~$\zeta \in \partial \mathbb{D}$, the non-tangential range $R_\Delta(f, \zeta)$ has two-dimensional Lebesgue measure zero.

{\em (iii)} For every $0 < \alpha < \pi$ and a.e.~$\zeta \in \partial \mathbb{D}$,
$$
\frac{m_2 \bigl (f(\Gamma_\alpha(\zeta)) \cap B(0,R) \bigr )}{\pi R^2} \to 0, \qquad \text{as }R \to \infty.
$$
\end{theorem}

The preceding example shows that at a.e.~Plessner point, the non-tangential range can have zero measure. Our second result says that it nevertheless has full Hausdorff dimension:

\begin{theorem}
\label{ntr-dim2}
Let $f: \mathbb{D} \to \mathbb{C}$ be a holomorphic function. For every $0 < \alpha < \pi$, at almost every Plessner point $\zeta \in \partial \mathbb{D}$, the set $R_{\alpha}(f, \zeta)$ has Hausdorff dimension 2.
\end{theorem}

\begin{remark}
(i) Recently, Gardiner and
Manolaki  \cite{gardiner-manolaki} showed that Plessner's theorem can be strengthened in a different direction. In particular, their result implies that at almost every
non-Fatou point $\zeta \in \partial \mathbb{D}$,
$$
\int_{\Gamma_\alpha(\zeta) \cap f^{-1}(B(w,r))} |f'(z)|^2\,dm_2(z)=\infty,
$$
 for every $0 < \alpha < \pi$ and ball $B(w, r) \subset \mathbb{C}$. Informally, this says that $f$ spends an infinite amount of time in every ball in the complex plane. Our example shows that $f$ can linger in balls without covering many points in them.
 
(ii) By looking at a super-lacunary series with exponents $2^{2^k}$ instead of $2^k$, one can show that the answer to Collingwood's question remains negative even if one replaces Stolz angles with horoballs.
\end{remark}

\subsection{Notation}

Let $0 < \alpha < \pi$. 
By a {\em Stolz angle} with vertex at $\zeta \in \partial \mathbb{D}$ and opening angle of $\alpha$, we mean 
$$
\Gamma_\alpha(\zeta) = \bigl \{ z \in \mathbb{D} \, : \,  |\zeta - z| < A(\alpha) \, (1-|z|) \bigr \}, \qquad A(\alpha) = \sec (\alpha/2).
$$
For $0 < h < 1$, we define the {\em truncated Stolz angle} as
$$
\Gamma_\alpha(\zeta, h) \, = \, \Gamma_\alpha(\zeta) \cap \{ 1 - h < |z| < 1 \}.
$$
Finally, for $\zeta \in \partial \mathbb{D}$ and $n \ge 0$, let 
$$V_\alpha(\zeta, n) = \Gamma_\alpha(\zeta) \cap \bigl \{1 - 2^{-n} \le |z| < 1 - 2^{-n-1} \bigr \}.$$
The sets $V_\alpha(\zeta, n)$ are pairwise disjoint and their union is all of $\Gamma_{\alpha}(\zeta)$.
When we have a particular value of $\alpha$ in mind, we omit it from the notation.

\section{Background in probability}
\label{sec:probability}

In this paper, a standard complex Gaussian has variance 1: for every measurable set $E \subset \mathbb{C}$,
\begin{equation}
\label{eq:gaussian-def}
\mathbb{P} ( \xi \in E ) \, = \, \frac{1}{\pi} \int_{E} e^{-|z|^2} \, dm_2(z).
\end{equation}
Likewise, we normalize planar Brownian motion so that $\mathbb{E}|B_t|^2=t$.
In this section, we collect several properties of Brownian motion that will be used below. 

While planar Brownian motion visits every neighbourhood of any fixed point in the plane, it almost surely misses the point itself. A classical theorem of Spitzer \cite[Theorem~1]{spitzer} quantifies how close it can approach. We will use the special case with $g(t)=\sqrt{t}\, e^{\beta t}$\,:

\begin{lemma}
\label{avoidance-lemma1}
Let $B_t$ be planar Brownian motion started at the origin. For every fixed $a \in \mathbb{C}$ and $\beta > 0$, almost surely,
\begin{equation*}
|B_t - a| > e^{-\beta t},
\end{equation*}
for all sufficiently large $t$. 
\end{lemma}

We now consider a Wiener sausage whose radius shrinks exponentially with time:

\begin{lemma}
\label{avoidance-lemma4}
Let $B_t$ be planar Brownian motion started at the origin. For any $\gamma > 0$, let
$$
W_\gamma = \bigcup_{t \ge 0} B(B_t, e^{-\gamma t}).
$$
Then, almost surely,
\begin{equation}
\label{eq:avoidance-lemma4}
\lim_{R \to \infty} \frac{m_2(W_\gamma \cap B(0,R))}{\pi R^2} = 0.
\end{equation}
\end{lemma}

\medskip

To prove the lemma, we split $W_\gamma$ into unit slices
$W_{\gamma,n} = \bigcup_{n \le t \le n+1} B(B_t, e^{-\gamma t})$, $n = 0, 1, 2, \dots$,
and estimate their areas separately. Since the radius varies only by a constant factor on each slice, it is natural to compare $W_{\gamma,n}$ with a Wiener sausage of constant radius. Let $K_\varepsilon$ be the $\varepsilon$-neighbourhood of $B([0,1])$. According to another classical result of Spitzer,
$$
\mathbb{E} m_2(K_\varepsilon) \sim \frac{\pi}{2 \log (1/\varepsilon)}, \qquad \text{as }\varepsilon \to 0,
$$
see \cite[p.~121]{spitzer2} or \cite[p.~167]{legall}.

\begin{lemma}
\label{one-unit-shrinking-sausage}
For $n \ge 1$ and $R \ge 1$, the expected area of the $n$-th slice of $W_\gamma$ intersected with $B(0,R)$ satisfies
$$
\mathbb{E} m_2 \bigl (W_{\gamma, n} \cap B(0,R) \bigr ) \lesssim
\begin{cases}
1/n, \qquad & n \le R^2, \\
R^2/n^2, \qquad & n > R^2.
\end{cases}
$$
\end{lemma}

\begin{proof}
Let 
$$
K_n = \bigcup_{0 \le s \le 1} B(B_{n+s} - B_n, \varepsilon_n), \qquad \varepsilon_n = e^{-\gamma n}.
$$
 Since $e^{-\gamma(n+s)}\le\varepsilon_n$ for $0\le s\le1$, we have $W_{\gamma,n} \subset B_n + K_n$.
As the increments of Brownian motion over disjoint time intervals are independent, $K_n$ is independent of $B_n$ and has the same law as $K_{\varepsilon_n}$. Therefore,
\begin{align*}
\mathbb{E} m_2 \bigl ((B_n + K_n) \cap B(0,R) \bigr ) & = \int_{\mathbb{C}} \mathbb{P}(u \in K_n) \, \mathbb{P} \bigl (B_n \in B(0,R) - u \bigr ) dm_2(u) \\
& \le \sup_{u \in \mathbb{C}} \Bigl \{ \mathbb{P} \bigl (B_n \in B(0,R) - u \bigr ) \Bigr \} \, \int_{\mathbb{C}} \mathbb{P}(u \in K_n) dm_2(u).
\end{align*}
By Spitzer's theorem above,
$$
\int_{\mathbb{C}} \mathbb{P}(u \in K_n) dm_2(u) \, = \, \mathbb{E} m_2(K_n) \, \asymp \, \frac{1}{\log(1/\varepsilon_n)} \, \asymp \, \frac{1}{n}.
$$
For the first term above, we consider two cases:
\begin{itemize}
\item If $n \le R^2$, we simply bound the probability that $B_n \in B(0,R) - u$ by 1.

\item If $n > R^2$, we use that the heat kernel is bounded by
$$
p_n(z) \, = \, \frac{1}{\pi n} \, e^{-|z|^2/n} \, \le \, \frac{1}{\pi n}.
$$
Therefore,
$$
\mathbb{P} \bigl ( B_n \in B(0,R) - u \bigr ) \, = \, \int_{B(0,R) - u} p_n(z) dm_2(z) \, \le \, \frac{m_2(B(0,R))}{\pi n} \, = \, \frac{R^2}{n}.
$$
\end{itemize}
The lemma follows after combining these estimates.
\end{proof}

\begin{proof}[Proof of Lemma \ref{avoidance-lemma4}]
By Lemma \ref{one-unit-shrinking-sausage}, for any $R \ge 1$,
$$
\mathbb{E}m_2(W_\gamma \cap B(0,R)) \, \lesssim \, 1 + \sum_{1 \le n \le R^2} \frac{1}{n} +  \sum_{n > R^2} \frac{R^2}{n^2} \, \lesssim \, \log(R+1).
$$
For $j \ge 0$, set
$$
X_j = \frac{m_2(W_\gamma \cap B(0,2^j))}{m_2(B(0,2^j))}.
$$
Then,
$$
\mathbb{E}X_j  \lesssim \frac{j+1}{4^j}
\qquad
\text{and}
\qquad
\mathbb{E} \biggl ( \sum_{j=0}^\infty X_j  \biggr ) < \infty.
$$
In particular, $X_j \to 0$ almost surely, i.e.~the quotient in (\ref{eq:avoidance-lemma4}) tends to 0 through $R = 2^j$. Monotonicity considerations show that the quotient tends to 0 through all $R$. The proof is complete.
\end{proof}

Finally, we need an estimate for the maximum displacement of planar Brownian motion. We first recall the corresponding one-dimensional bound.
Let $W_t$ be one-dimensional Brownian motion with variance $t/2$. By the reflection principle and the Gaussian tail estimate,
\begin{equation}
\label{eq:reflection-principle}
\mathbb{P} \Bigl ( \max_{0 \le s \le t} |W_s| > a \Bigr )
 \le
2 \, \mathbb{P} \Bigl ( \max_{0 \le s \le t} W_s > a \Bigr )
 =
4 \, \mathbb{P} \bigl ( W_t > a \bigr )
\le
2 \exp \biggl ( - \frac{a^2}{t} \biggr),
\end{equation}
see \cite[Theorem~2.21]{morters-peres}.
\begin{lemma}
\label{maximum-bm}
Let $B_t$ be planar Brownian motion started at the origin and $M_t = \max_{0 \le s \le t} |B_s|$. Then,
$$
\mathbb{P}(M_t > r) \le 4 \exp \biggl ( - \frac{r^2}{2t} \biggr ).
$$
\end{lemma}

\begin{proof}
Write $B_t = X_t + iY_t$. Since $\mathbb{E}|B_t|^2=t$, the real and imaginary parts $X_t$ and $Y_t$ are one-dimensional Brownian motions with variance $t/2$.
If $|B_s| > r$, then either
$$
|X_s| > \frac{r}{\sqrt{2}} \qquad \text{or} \qquad |Y_s| > \frac{r}{\sqrt{2}}.
$$
Therefore, by (\ref{eq:reflection-principle}),
$$
\mathbb{P}(M_t > r)
\, \le \,
\mathbb{P} \biggl ( \max_{0 \le s \le t} |X_s| > \frac{r}{\sqrt{2}} \biggr )
+
\mathbb{P} \biggl ( \max_{0 \le s \le t} |Y_s| > \frac{r}{\sqrt{2}} \biggr )
\, \le \,
4 \exp \biggl ( - \frac{r^2}{2t} \biggr ),
$$
as desired.
\end{proof}

\section{Basic properties}
\label{sec:basic-properties}

Let $f$ be the random lacunary series (\ref{eq:f-def}).  We work on $\Omega = \mathbb{C}^{\mathbb{N}}$, equipped with the product
$\mathbb{P}$ of standard complex Gaussian measures.  A point $\omega = (\xi_k)_{k=1}^\infty \in \Omega$ corresponds to a particular sequence of coefficients. Taking $r = k^{1/4}$ in the Gaussian tail estimate
$$
\mathbb{P} \bigl ( |\xi_k| > r \bigr ) \, = \, \frac{1}{\pi} \int_{|z| > r} e^{-|z|^2} \, dm_2(z) \, = \, e^{-r^2},
$$
and applying the Borel-Cantelli lemma, we conclude that almost surely,
\begin{equation}
\label{eq:bound-on-coefficients}
\frac{|\xi_k|}{\sqrt{k}} \le k^{-1/4}, \qquad \text{for all }k \ge k_0(\omega).
\end{equation}

\subsection{Shadowing property}

We denote the $n$-th partial sum of (\ref{eq:f-def}) at the point $e^{i\theta} \in \partial \mathbb{D}$ by
\begin{equation}
\label{eq:Sn-def}
S_n(\theta) \, = \, \sum_{k=1}^n \frac{\xi_k}{\sqrt{k}} \, e^{i 2^k \theta}, \qquad \theta \in [0, 2\pi),
\end{equation}
with the convention that $S_0(\theta)=0$. The following lemma says that on each dyadic piece of a Stolz angle, $f$ is uniformly close to the corresponding partial sum $S_n(\theta)$\,:

\begin{lemma}
\label{f-is-close-to-Sn}
Suppose the coefficients $(\xi_k)_{k=1}^\infty$ satisfy (\ref{eq:bound-on-coefficients}).

{\em (i)} The power series (\ref{eq:f-def}) defining $f$ converges uniformly and absolutely on compact subsets of the unit disk and defines a holomorphic function.

{\em (ii)} There exists a constant $C = C(\alpha,\omega)$ such that for every $\theta \in [0, 2\pi)$, $n \ge 0$ and $z \in V_\alpha(e^{i\theta}, n)$,
$$
|f(z) - S_n(\theta)| \le C (n+1)^{-1/4}.
$$
\end{lemma}

\begin{proof}
(i) From the bound on the size of the coefficients, it follows that the radius of convergence of (\ref{eq:f-def}) is at least 1.

(ii) Let $\zeta = e^{i\theta}$. For $k \le n$, we have
$$
\bigl | z^{2^k} - \zeta^{2^k} \bigr | \, \le \, 2^k |z - \zeta| \, \le \, A(\alpha) \cdot 2^{k-n}.
$$
For $k > n$, the estimate $1 - x \le e^{-x}$, $0 < x < 1$, gives
$$
\bigl |z^{2^k} \bigr | \le (1 - 2^{-n-1})^{2^k} \, \le \, \exp \bigl (-2^{-n-1} \cdot 2^k \bigr ) \, = \, \exp (- 2^{k-n-1}).
$$
Therefore,
$$
|f(z) - S_n(\theta)| \lesssim \biggl ( \sum_{k \le n} k^{-1/4} 2^{k-n} + \sum_{k > n} k^{-1/4} e^{-2^{k-n-1}} \biggr ) = O \bigl ((n+1)^{-1/4} \bigr ),
$$
as desired. The finitely many exceptional coefficients in (\ref{eq:bound-on-coefficients}) are absorbed into the implicit constant.
\end{proof}

\subsection{Embedding the partial sums in a Brownian motion}
\label{sec:brownian-embedding}

For each fixed $\theta \in [0,2\pi)$, the increments
$$
S_{n+1}(\theta)-S_n(\theta) = (\xi_{n+1}/\sqrt{n+1}) \, e^{i 2^{n+1} \theta}
$$
are independent complex Gaussians with variance $1/(n+1)$. In particular, the law of $(S_n(\theta))$ does not depend on $\theta$. Set
$$
H_n=\sum_{k=1}^n \frac{1}{k}.
$$
Since $S_0(\theta)=0$ and $H_{n+1}-H_n=1/(n+1)$, we can embed the process $(S_n(\theta))$ into a planar Brownian motion $(B_t)$ started at the origin so that
$$
S_n(\theta)=B_{H_n}.
$$
As $H_n=\log n+O(1)$, the $n$-clock runs exponentially faster than the Brownian clock.

\subsection{\texorpdfstring{$S_n(\theta)$ does not come too close to a fixed point}{Sₙ(θ) does not come too close to a fixed point}}

The following lemma says that for fixed $\theta \in [0, 2\pi)$ and $a \in \mathbb{C}$, the partial sums $S_n(\theta)$ do not approach $a$ too rapidly:

\begin{lemma}
\label{far-away-from-a}
For any given pair $(\theta, a) \in [0,2\pi) \times \mathbb{C}$, almost surely,
$$
|S_n(\theta) - a| > n^{-1/8},
$$
for all sufficiently large $n$.
\end{lemma}

\begin{proof}
Since $H_n = \log n + O(1)$, we may apply Lemma \ref{avoidance-lemma1} with $\beta = 1/16$ to obtain
$$
|S_n(\theta) - a| \, > \, \exp \biggl (-\frac{1}{16} \bigl (\log n + O(1) \bigr ) \biggr ) \,\ge\, C n^{-1/16} \, > \, n^{-1/8},
$$
for all sufficiently large $n$.
\end{proof}

\section{A Fubini argument}
\label{sec:fubini}

Since we will repeatedly use the following simple Fubini argument, we take a moment to formalize the setup.
We equip $[0, 2\pi) \times \Omega$ with the product probability measure $\mathcal P = \frac{d\theta}{2\pi} \otimes \mathbb{P}$.
Suppose that $E \subset [0, 2\pi) \times \Omega$ is a measurable set. We denote the horizontal and vertical slices of $E$ by
$$
E_\omega = \bigl  \{ \theta \in [0,2\pi) : (\theta, \omega) \in E \bigr  \}, \qquad \omega \in \Omega
$$
and
$$
E^\theta = \bigl \{ \omega \in \Omega : (\theta, \omega) \in E \bigr \}, \qquad \theta \in [0,2\pi)
$$
respectively.

\begin{lemma}
\label{fubini}
If $\mathbb{P}(E^\theta) = 1$ for every $\theta$, then  $|E_\omega| = 2\pi$ for a.e.~$\omega$.
\end{lemma}

\begin{proof}
By Fubini's theorem,
$$
 \frac{1}{2\pi} \int_\Omega  |E_\omega| \,  d\mathbb{P}(\omega) \, = \, \mathcal P(E) \, = \, \frac{1}{2\pi} \int_0^{2\pi} \mathbb{P}(E^\theta) d\theta \, = \, 1.
$$
Hence,
 $|E_\omega| = 2\pi$ for a.e.~$\omega$. 
\end{proof}

We now give a sample application illustrating how the lemma will be used in the paper. Suppose that for each $\theta \in [0,2\pi)$ and $\omega \in \Omega$, we have a measurable set $Q(\theta,\omega) \subset \mathbb C$ and can show that for every fixed $\theta$, 
$$m_2(Q(\theta, \omega)) = 0, \qquad \text{almost surely}.$$
 Then, almost surely,
$$m_2(Q(\theta, \omega)) = 0, \qquad \text{for almost every }\theta.$$

\section{The Fatou set is almost empty}
\label{sec:fatou-empty}

In this section, we show Theorem \ref{main-thm}(i), which says that almost surely,
\begin{equation}
\label{eq:oscillation}
\liminf_{r \to 1} |f(re^{i\theta})| = 0, \qquad \limsup_{r \to 1} |f(re^{i\theta})| = \infty,
\end{equation}
for a.e.~$\theta \in [0, 2\pi)$.

\begin{proof}[Proof of Theorem \ref{main-thm}{\em (i)}]
{\em Step 1.} Fix $\theta \in [0,2\pi)$. Recall that the process $(S_n(\theta))$ has the same law as $(B_{H_n})$ where $H_n = \sum_{k=1}^n 1/k$.
Since planar Brownian motion is almost surely recurrent and unbounded,
$$
\liminf_{t \to \infty} |B_t| = 0, \qquad \limsup_{t \to \infty} |B_t| = \infty. 
$$
We show that both conclusions remain valid along the discrete sequence of times $H_n$. Set
$$
\osc_n :=  \sup_{H_n \le t \le H_{n+1}} |B_t - B_{H_n}|.
$$
As $H_{n+1}-H_n=1/(n+1)$, the Markov property of Brownian motion and Lemma \ref{maximum-bm} yield
$$
\mathbb P\bigl(\osc_n>n^{-1/3}\bigr)
\, \le \,
4\exp \biggl (-\frac{n+1}{2n^{2/3}} \biggr)
\, \le \,
4\exp\biggl (-\frac{n^{1/3}}{2} \biggr).
$$
Since the right hand side is summable, by the Borel-Cantelli lemma, $\osc_n \to 0$ almost surely. As the intervals
$[H_n,H_{n+1}]$ cover $[H_1,\infty)$, it follows that almost surely,
$$
\liminf_{n \to \infty} |B_{H_n}| = 0, \qquad \limsup_{n \to \infty} |B_{H_n}| = \infty.
$$
In other words, for each fixed $\theta \in [0,2\pi)$, almost surely
$$
\liminf_{n \to \infty} |S_n(\theta)| = 0, \qquad \limsup_{n \to \infty} |S_n(\theta)| = \infty.
$$

{\em Step 2.} 
Applying Fubini's theorem as in Section \ref{sec:fubini}, we conclude that almost surely
$$
\liminf_{n \to \infty} |S_n(\theta)| = 0, \qquad \limsup_{n \to \infty} |S_n(\theta)| = \infty,
$$
for a.e.~$\theta \in [0,2\pi)$.

\medskip

{\em Step 3.}
Set $r_n = 1 - 2^{-n}$.
By Lemma \ref{f-is-close-to-Sn}, almost surely
$$
|f(r_n e^{i\theta}) - S_n(\theta)| \to 0, \qquad \text{as }n \to \infty,
$$
uniformly in $\theta$. As a result, if $|S_n(\theta)|$ oscillates between 0 and $\infty$, then so does $f$ along the sequence
$r_n e^{i\theta}$. This proves (\ref{eq:oscillation}).
\end{proof}

\section{Non-tangential ranges have zero area}
\label{sec:footprint}

In this section, we show Theorem \ref{main-thm}(ii), which says that for a.e.~$\theta \in [0,2\pi)$, the non-tangential range $R_\Delta(f, e^{i\theta})$ has two-dimensional Lebesgue measure zero.

\begin{proof}[Proof of Theorem \ref{main-thm}{\em (ii)}]
It suffices to show that for any $0 < \alpha < \pi$, almost surely $m_2(R_\alpha(f, e^{i\theta})) = 0$ for a.e.~$\theta \in [0,2\pi)$.

\medskip

{\em Step 1.} Let $L_\theta = \bigcap_{m \ge 1} \bigcup_{n \ge m} B\bigl(S_n(\theta), n^{-1/8}\bigr)$ be the set of points in the plane that lie in infinitely many of the balls 
\begin{equation}
\label{eq:the-balls}
B(S_n(\theta), n^{-1/8}), \qquad n = 1,2, \dots.
\end{equation}
By Lemma \ref{far-away-from-a},
$
\mathbb P(a\in L_\theta)=0
$
for any fixed pair $(\theta,a) \in [0,2\pi) \times \mathbb{C}$.

\medskip

{\em Step 2.}
Fix a $\theta \in [0, 2\pi)$ and an $R > 0$. By Fubini's theorem,
\begin{align*}
\mathbb{E} \bigl [ m_2(L_\theta \cap B(0,R)) \bigr ]
& = \mathbb{E} \int_{B(0,R)} \chi_{\{a \in L_\theta\}} dm_2(a) \\ 
& = \int_{B(0,R)} \mathbb{P}  ( a \in L_\theta   )dm_2(a) \\
& = 0.
\end{align*}
Thus, for every fixed $\theta \in [0, 2\pi)$ and $R > 0$, almost surely $m_2(L_\theta \cap B(0,R)) = 0$. Intersecting these almost sure events over $R\in\mathbb N$, we conclude that for every fixed $\theta$, almost surely $m_2(L_\theta)=0$.

\medskip

{\em Step 3.} Fubini's theorem as in Section \ref{sec:fubini} shows that almost surely,
\begin{equation}
\label{eq:bound-on-measure2}
m_2(L_\theta) = 0, \qquad \text{for a.e.~}\theta \in [0,2\pi).
\end{equation}

\medskip

{\em Step 4.} Fix a sequence of coefficients $(\xi_k)_{k=1}^\infty$ for which both (\ref{eq:bound-on-coefficients}) and (\ref{eq:bound-on-measure2}) hold. For $a \in R_\alpha(f,e^{i\theta})$, there exists a sequence $z_j \in \Gamma_\alpha(e^{i\theta})$ converging to $e^{i\theta}$ such that $f(z_j)=a$. Each point $z_j$ lies in some dyadic piece $V_{\alpha}(e^{i\theta}, n_j)$. Since $z_j \to e^{i\theta}$, we have $n_j \to \infty$.
Lemma \ref{f-is-close-to-Sn} implies that
$$
|S_{n_j}(\theta)-a|
\le C(\alpha, \omega) (n_j+1)^{-1/4}<n_j^{-1/8},
$$
for all sufficiently large $j$.
Thus, $a$ belongs to infinitely many of the balls (\ref{eq:the-balls}) and hence $a \in L_\theta$. Therefore,
\begin{equation}
\label{eq:R-inclusion}
R_\alpha(f, e^{i\theta})\subset L_\theta,
\qquad
\text{for all }\theta \in [0,2\pi).
\end{equation}
Combining this inclusion with Step 3, we obtain that $m_2(R_\alpha(f, e^{i\theta})) = 0$ for almost every $\theta$. The proof is complete.
\end{proof}
 
\section{Images of Stolz angles are sparse near infinity}
\label{sec:sparse-near-infinity}

We say that a set $E \subset \mathbb{C}$ has {\em asymptotic density zero} if
$$
\frac{m_2 \bigl (E \cap B(0,R) \bigr )}{\pi R^2} \to 0, \qquad \text{as }R \to \infty.
$$
In this section, we show Theorem \ref{main-thm}(iii), which says that almost surely,
$f(\Gamma_\alpha(e^{i\theta}))$ has asymptotic density zero for every $0 < \alpha < \pi$ and a.e.~$\theta \in [0,2\pi)$.

\begin{proof}[Proof of Theorem \ref{main-thm}{\em (iii)}]
{\em Step 1.}
By Lemma \ref{avoidance-lemma4}, almost surely, the shrinking Wiener sausage
$$
W_{1/16} = \bigcup_{t \ge 0} B(B_t, e^{-t/16})
$$
has asymptotic density zero. Fix $\theta \in [0, 2\pi)$ and as usual, identify the process $(S_n(\theta))$ with $(B_{H_n})$.
Since $H_n = \log n + O(1)$, there exists an absolute constant $N_1 > 0$ such that
$$
n^{-1/8}\le e^{-H_n/16}, \qquad n \ge N_1.
$$
Hence,
$$
U_\theta \, := \, \bigcup_{n \ge N_1} B(S_n(\theta), n^{-1/8}) \, \subset \, W_{1/16}.
$$
Therefore, for every fixed $\theta \in [0,2\pi)$, almost surely, the set $U_\theta$ also has asymptotic density zero.
Applying Fubini's theorem as in Section \ref{sec:fubini}, we conclude that almost surely,
\begin{equation}
\label{eq:sparse-near-infinity2}
\frac{m_2(U_\theta \cap B(0,R))}{\pi R^2} \to 0, \qquad \text{for almost every }\theta.
\end{equation}

\medskip

{\em Step 2.}
Fix a sequence of coefficients $(\xi_k)$ satisfying both (\ref{eq:bound-on-coefficients}) and (\ref{eq:sparse-near-infinity2}). Let $\theta \in [0,2\pi)$ be an angle for which (\ref{eq:sparse-near-infinity2}) holds.
By Lemma \ref{f-is-close-to-Sn}, there exists a constant $C = C(\alpha, \omega)$ such that
$$
|f(z) - S_n(\theta)| \le C (n+1)^{-1/4}, \qquad z \in V_\alpha(e^{i\theta}, n), \quad n \ge 0.
$$
Choose $N_2 = N_2(\alpha, \omega)$ so that $C(n+1)^{-1/4} < n^{-1/8}$ for $n \ge N_2$ and set $N = \max(N_1, N_2)$.
Since
$$f \biggl (\bigcup_{n \ge N} V_\alpha(e^{i\theta}, n) \biggr ) \subset U_\theta,$$
this set has asymptotic density zero.
On the other hand,
$\bigcup_{n<N}V_\alpha(e^{i\theta},n)$
is compactly contained in the unit disk, and hence its image under $f$ is bounded. In particular,
$$
f\biggl(\bigcup_{n<N}V_\alpha(e^{i\theta},n)\biggr)
$$
also has asymptotic density zero. We conclude that their union $f(\Gamma_\alpha(e^{i\theta}))$ has asymptotic density zero.
\end{proof}

\section{The non-tangential range has dimension 2}
\label{sec:non-tangential-range}

In this section, we prove Theorem \ref{ntr-dim2}.
We will use the following classical lemma due to Frostman:

\begin{lemma}
\label{cap-0a}
Suppose that a holomorphic function  $f: \mathbb{D} \to \mathbb{C}$ omits a compact set $L \subset \mathbb{C}$ of positive logarithmic capacity. Then, $f$ belongs to the Nevanlinna class $\mathcal N$. In particular, $f$ has finite non-tangential limits a.e.~on the unit circle.
\end{lemma}

\begin{proof}
Since the image of $f$ is connected and disjoint from $L$, it is contained in a
single connected component $\Omega$ of $\mathbb{C} \setminus L$. If $\Omega$ is
bounded, then $f \in H^\infty$ and there is nothing to prove. We may therefore
assume that $\Omega$ is the unbounded component.

Since $L$ has positive capacity, $\Omega$ possesses a Green's function. By \cite[Theorem 5.2.1]{ransford},
$$
 G_{\Omega}(w, \infty) = \log |w| - \log \capacity(L) + o(1), \qquad \text{as }w \to \infty,
$$
where $\capacity(L)$ denotes the logarithmic capacity of $L$. As $G_\Omega(\cdot, \infty)$ is non-negative, there exists a constant $C > 0$ such that
$
\log^+|w| \le G_{\Omega}(w, \infty) + C.
$
Consequently, $$\log^+|f(z)| \le G_{\Omega}(f(z), \infty) + C.$$ Since the right hand side is a harmonic function, $\log^+|f|$ admits a harmonic majorant and so $f\in\mathcal N$. 
\end{proof}

\begin{corollary}
\label{cap-0}
Let $L \subset \mathbb{C}$ be a compact set of positive logarithmic capacity. Then, for a.e.~Plessner point $\zeta \in \partial \mathbb{D}$, 
$$
f(\Gamma_\alpha(\zeta, h)) \cap L \ne \varnothing,
$$
for every $0 < \alpha < \pi$ and $0 < h < 1$.
\end{corollary}

\begin{proof}
Otherwise, there exists a set $E \subset \partial \mathbb{D}$ of positive measure, consisting entirely of Plessner points, such that for every $\zeta\in E$, the function $f$ omits $L$ on some truncated Stolz angle
$\Gamma_{\alpha(\zeta)}(\zeta,h(\zeta))$. By restricting to a subset of positive measure, we may assume that the parameters $\alpha=\alpha(\zeta)$ and $h=h(\zeta)$ are independent of $\zeta$. Using inner regularity and restricting further if necessary, we may also assume that $E$ is compact and contained in a short arc, so that the sawtooth region
$$
\Omega = \bigcup_{\zeta\in E}\Gamma_\alpha(\zeta,h)
$$
is a Jordan domain with rectifiable boundary.

 Let $\varphi:\mathbb{D}\to\Omega$ be a Riemann map. By Carath\'eodory's theorem, $\varphi$ extends to a homeomorphism between the closures. Since $\partial\Omega$ is rectifiable, the F.~and M.~Riesz theorem \cite[Theorem VI.1.1]{garnett-marshall} implies that the boundary correspondence is absolutely continuous. In particular, $\varphi^{-1}(E)$ has positive Lebesgue measure. Moreover, $\varphi$ has a finite non-zero angular derivative at almost every point of the unit circle. Since $f\circ\varphi$ omits $L$, Lemma \ref{cap-0a} implies that it has a finite non-tangential limit at a.e.~point of the unit circle. Choose a point $\xi\in\varphi^{-1}(E)$ at which $\varphi$ has a finite non-zero angular derivative and $f\circ\varphi$ has a finite non-tangential limit. As explained in \cite[Chapter V.5]{garnett-marshall}, $\varphi$ is then conformal (isogonal) at $\xi$ and carries Stolz angles at $\xi$ to slightly distorted Stolz angles at
$\zeta=\varphi(\xi)$. It follows that $f$ has a finite non-tangential limit at $\zeta$, contradicting the assumption that $\zeta$ is a Plessner point.
\end{proof}

\begin{proof}[Proof of Theorem \ref{ntr-dim2}]
Fix $0 < \alpha < \pi$ and let $P(f)$ be the set of Plessner points of $f$. Recall that
$$
I_n(\zeta) = f(\Gamma_\alpha(\zeta, 2^{-n})), \qquad R_\alpha(f, \zeta) = \bigcap_{n \ge 1} I_n(\zeta).
$$
Fix $0 < d < 1$ and let $K$ be a self-similar Cantor set generated by the contractions
$$
\phi_0(x) = \lambda x, \qquad \phi_1(x) = 1 - \lambda x, \qquad \lambda \, = \, 2^{-1/d} \, < \, 1/2.
$$
By the choice of $\lambda$, the Hausdorff and Minkowski dimensions of $K$ are both equal to $d$. 
For a finite word $\xi = \xi_1 \xi_2 \dots \xi_k \in \{0,1\}^*$, we denote by 
$$
K_\xi = \phi_{\xi_1} \circ \phi_{\xi_2} \circ \dots \circ \phi_{\xi_k}(K)
$$
 the cylinder subset of $K$ corresponding to $\xi$.
By self-similarity, $$\Hdim K_\xi \, = \, \Mdim K_\xi \, = \, d$$ as well. In particular, every cylinder set $K_\xi$ has positive logarithmic capacity.

Fix $w \in \mathbb{C}$. By Corollary \ref{cap-0}, for a.e.~$\zeta \in P(f)$,
$$
I_n(\zeta) \cap (w + K_\xi) \ne \varnothing,
$$
for every $n$ and every cylinder $K_\xi$. It follows that for every $n$,
$$
I_n(\zeta)  \cap (w + K)
$$
is a relatively open and dense subset of $w + K$. Therefore, by Baire's category theorem, $R_\alpha(f, \zeta) \cap (w + K) \ne \varnothing$
for a.e.~$\zeta \in P(f)$.

Applying Fubini's theorem in the product space $(P(f), m_1) \times (\mathbb{C}, m_2)$, we conclude that for a.e.~$\zeta \in P(f)$,
$$
R_\alpha(f, \zeta) \cap (w + K) \ne \emptyset, \qquad \text{for a.e.~}w \in \mathbb{C}.
$$
Equivalently, the complement of $R_\alpha(f,\zeta)-K$ has measure zero. In particular,
$\Hdim (R_\alpha(f,\zeta)-K )=2$.

Since the map $(x, y) \to x - y$ is Lipschitz,
$$
\Hdim \, (R_\alpha(f, \zeta) -K) \le \Hdim \, (R_\alpha (f, \zeta) \times K).
$$
By Marstrand's product theorem, e.g.~see \cite[Equation (3.2.2)]{bishop-peres},
\begin{align*}
\Hdim \, (R_\alpha (f, \zeta) \times K)
& \le \Hdim R_\alpha(f, \zeta) + \Mdim K \\
& = \Hdim R_\alpha(f, \zeta) + d.
\end{align*}
Combining these inequalities, we get $\Hdim R_\alpha(f, \zeta) \ge 2 - d$ for a.e.~$\zeta \in P(f)$. 
Letting $d \to 0$, we obtain $\Hdim R_\alpha(f,\zeta)=2$ for a.e.~$\zeta\in P(f)$. The proof is complete.
\end{proof}

\subsection*{Acknowledgements} This research was supported by the Israel Science Foundation
(grant 3134/21).

\subsection*{Declaration on the use of AI Tools}
The proofs of Theorems 1.1 and 1.2 were found by ChatGPT-6 Astra. The role of the human author is purely expository.

\end{document}